\documentclass{amsart}
\newtheorem{thr}{Theorem}

\newtheorem{con}[thr]{Conjecture}
\newtheorem{claim}[thr]{Claim}

\theoremstyle{definition}

\theoremstyle{remark}

\numberwithin{equation}{section}

\def\un{{(1)}}
\def\du{{(0)}}
\def\G{\mathcal{G}}

\def\S{\mathcal{S}}

\def\ep{\varepsilon}
\def\adj{\operatorname{adj}}

\begin{document}

\title[On the characteristic polynomial of a supertropical adjoint matrix]{On the characteristic polynomial \\ of a supertropical adjoint matrix}

\author{Yaroslav Shitov}
\address{National Research University Higher School of Economics, 20 Myasnitskaya Ulitsa, Moscow 101000, Russia}
\email{yaroslav-shitov@yandex.ru}

\subjclass[2000]{15A80}
\keywords{tropical algebra, matrix theory}

\begin{abstract}
Let $\chi(A)$ denote the characteristic polynomial of a matrix $A$ over a field;
a standard result of linear algebra states that $\chi(A^{-1})$ is the reciprocal
polynomial of $\chi(A)$. More formally, the condition $\chi^n(X) \chi^k(X^{-1})=\chi^{n-k}(X)$
holds for any invertible $n\times n$ matrix $X$ over a field, where $\chi^i(X)$ denotes
the coefficient of $\lambda^{n-i}$ in the characteristic polynomial $\det(\lambda I-X)$.
We confirm a recent conjecture of Niv by proving the tropical analogue of this result.
\end{abstract}

\maketitle

The supertropical semifield is a relatively new concept arisen as a tool for studying problems
of tropical mathematics~\cite{IRIR2}. The supertropical theory is now a developed branch of algebra,
and we refer the reader to~\cite{IRIR} for a survey of basics and applications.
Our arguments make use of some other structures including fields and polynomial rings over them, so it will
be convenient for us to work with slightly unusual equivalent description of supertropical semifield.
In particular, we will denote the operations by $\oplus$ and $\odot$ to avoid confusion with standard
operations $+$ and $\cdot$ over a field. For the same reason, we will use the notation $u^{\odot i}$ in supertropical
setting while $u^i$ will denote a power of element of a field. Similarly, we will denote the supertropical determinant
by $\det_\circ$ reserving the notation $\det$ for usual determinant over a field.
Let us recall the definitions of concepts mentioned above.

Let $(\G,\ast,0,\leq)$ be an ordered Abelian group, and $\G^\du$ and $\G^\un$ be two copies of $\G$.
We consider the semiring $\S= \G^\du\cup \G^\un \cup\{\ep\}$ with two commutative operations,
denoted by $\oplus$ and $\odot$. Assume $i,j\in\{0,1\}$, $s\in\S$, $a,b\in \G$, and let $a<b$; the operations
are defined by $\ep\oplus s=s$, $\ep\odot s=\ep$, $a^{(i)}\oplus b^{(j)}=b^{(j)}$, $b^{(i)}\oplus b^{(j)}=b^\du$,
$a^{(i)}\odot b^{(j)}=(a\ast b)^{(ij)}$. One can note that $\S$ is isomorphic to \textit{supertropical semifield}, and
the elements from $\G^\du$ and $\G^\un$ correspond to \textit{ghost} and \textit{tangible} elements,
respectively. We define the mapping $\nu$ sending $a^{(i)}$ to $a\in\G$ and $\varepsilon$ to $\varepsilon$;
we say that $c,d\in\S$ are $\nu$\textit{-equivalent} whenever $\nu(c)=\nu(d)$, and we write $c\approx_\nu d$ in this case.
Also, we write $c\models d$ if either $c=d$ or $c=d+g$,
for some ghost element $g$; this relation is known as \textit{ghost surpassing} relation,
which is one of fundamental concepts replacing equality in many theorems taken from
classical algebra~\cite{IRIR}. By $u^{\odot i}$ we denote the $i$th supertropical
power of $u$, that is, the result of multiplying $u$ by itself $i$ times. Let $A=(a_{ij})$ be a supertropical matrix;
its determinant is
$$\operatorname{det}_{\circ} A=\bigoplus_{\sigma\in S_n}a_{1\sigma(1)}\odot\ldots\odot a_{n\sigma(n)},$$
where $S_n$ denotes a symmetric group on $\{1,\ldots,n\}$.
The matrix $A$ is said to be \textit{non-singular} if $\det_\circ A$ is tangible; equivalently,
$A$ is non-singular if $\det_\circ A$ has a multiplicative inverse in $\S$.
The $(i,j)$th \textit{cofactor} of $A$ is the supertropical determinant
of the matrix obtained from $A$ by removing $i$th row and $j$th column.
By $\operatorname{adj}_\circ A$ we denote the \textit{adjoint} of $A$, that is,
the $n\times n$ matrix whose $(i,j)$th entry equals the $(j,i)$th cofactor.
By $\chi^k_\circ(A)$ we denote the supertropical sum of all principal $k\times k$ minors of $A$, that is,
the coefficient of $\lambda^{\odot(n-k)}$ in the \textit{characteristic polynomial} $\det_\circ(A\oplus\lambda\odot I_\circ)$.
Note that $\varepsilon$ and $0^\un$ are neutral elements with respect to $\oplus$ and $\odot$, respectively; therefore, the
\textit{supertropical identity matrix} $I_\circ$ has elements $0^\un$ on diagonal and $\varepsilon$'s everywhere else.
The following has been an open problem. 

%

\begin{con}\cite[Conjecture~6.2]{Niv}\label{thrConjNiv}
Let $A\in\S^{n\times n}$ be a non-singular matrix. Then, $\chi^k_\circ(\operatorname{adj}_\circ A)\models\left(\det_\circ A \right)^{\odot(k-1)}\odot  \chi^{n-k}_\circ(A)$ holds for all $k\in\{0,\ldots,n\}$.
\end{con}

Actually, Niv formulates this conjecture in a slightly different but equivalent way.
If $A$ is tropically non-singular, then its \textit{pseudoinverse} is defined as
$A^\nabla=(\det A)^{\odot(-1)}\odot (\adj_\circ A)$. Multiplying both sides of equality in Conjecture~\ref{thrConjNiv}
by $(\det A)^{\odot(1-k)}$, one gets $(\det_\circ A)\odot \chi^k_\circ(A^\nabla)\models \chi^{n-k}_\circ(A)$,
which is exactly the formulation given by Niv. To prove Conjecture~\ref{thrConjNiv}, we consider the $n\times n$ matrix
$V$ consisting of variables $(v_{ij})$, and we define polynomials $\alpha,\beta\in\S[V]$ as
$$\alpha=\chi^k_\circ(\operatorname{adj}_\circ V),\,\,\,\,\,\beta=\left(\operatorname{det}_\circ V \right)^{\odot(k-1)}\odot  \chi^{n-k}_\circ(V).$$
Note that any coefficient of $\alpha$ and $\beta$ is either $0^{(0)}$ or $0^{(1)}$; define $\gamma\in\S[V]$
as the supertropical sum of those monomials that appear in $\beta$ with tangible coefficients.

\begin{claim}\label{clcl3}
If $A\in\S^{n\times n}$ is a non-singular matrix, then $\beta(A)=\gamma(A)$.
\end{claim}

\begin{proof}
By definition, $\beta$ is supertropical sum of monomials $m_\mu=m^1_\mu\odot m^2_\mu$, where
$$m^1_\mu=\left(\bigodot_{i=1}^n v_{i\sigma_1(i)}\right)\odot\ldots\odot\left(\bigodot_{i=1}^n v_{i\sigma_{k-1}(i)}\right),
\,\,\,\,\,m^2_\mu=\bigodot_{j\in J} v_{j\tau(j)},$$
over all tuples $\mu=(\sigma_1,\ldots,\sigma_{k-1},J,\tau)$ such that $\sigma_1,\ldots,\sigma_{k-1}\in S_n$,
a subset $J\subset\{1,\ldots,n\}$ has $n-k$ elements, and $\tau$ is a permutation of $J$.

Assume $\beta(A)\neq\varepsilon$. Suppose that there exist distinct tuples $\mu$ and
$\mu'=(\sigma'_1,\ldots,\sigma'_{k-1},J',\tau')$ satisfying $m_\mu= m_{\mu'}$ and $\beta(A)\approx_\nu m_\mu(A)$.
Then, since there is a unique permutation $\sigma$ such that $\bigodot_{i=1}^n a_{i\sigma(i)}\approx_\nu\det_\circ A$,
we have that all $\sigma_t$ and all $\sigma_{t}'$ are equal to $\sigma$. In particular, we have $m^1_{\mu}=m^1_{\mu'}$,
which implies $m^2_{\mu}=m^2_{\mu'}$. Note that the latter condition in turn implies $J=J'$ and $\tau=\tau'$, a contradiction.

Therefore, either $\beta(A)=\varepsilon$ or $\nu(m(A))\neq\nu(\beta(A))$ holds for any monomial $m$
that appears in $\beta$ with ghost coefficient. This means that we can remove all these monomials
from $\beta$ without changing the value of $\beta(A)$.
\end{proof}

\begin{claim}\label{clcl1}
If a monomial $v_{11}^{\odot k_{11}}\odot\ldots\odot v_{nn}^{\odot k_{nn}}$ appears with a tangible coefficient in either $\alpha$ or $\beta$,
then it appears in both $\alpha$ and $\beta$ with coefficients different from $\ep$.
\end{claim}

\begin{proof}
Let $X=(x_{ij})$ be a matrix whose entries are variables of
the polynomial ring $\mathbb{C}[x_{11},\ldots,x_{nn}]$, and
define $\varphi=\chi^k(\operatorname{adj} X)$, 
$\psi=\left(\det X \right)^{k-1} \chi^{n-k}(X)$. Let us get rid
of brackets by distributivity in the standard expressions of
$\varphi$ and $\psi$ and denote the expressions we obtain before
canceling terms by $\varphi_0$ and $\psi_0$, respectively.
Now, if we replace any monomial $\pm x_{11}^{ k_{11}} \ldots x_{nn}^{ k_{nn}}$ by
$0^{(1)}\odot v_{11}^{\odot k_{11}}\odot\ldots\odot v_{nn}^{\odot k_{nn}}$ in
$\varphi_0$ and $\psi_0$, we get $\alpha(V)$ and $\beta(V)$.
Since the equality $\varphi=(-1)^n \psi$ is true for matrices over a field, the total
number of appearances in $\varphi_0$ and $\psi_0$ is even for any monomial.
\end{proof}

\begin{claim}\label{clcl11}
Let $A\in\S^{n\times n}$ be a non-singular matrix. If $\alpha(A)$ is a tangible element, then there is $s\in\S$
such that $\beta(A)=\alpha(A)\oplus s$.
\end{claim}

\begin{proof}
By assumption, there is a monomial $m=v_{11}^{\odot k_{11}}\odot\ldots\odot v_{nn}^{\odot k_{nn}}$ appearing in $\alpha$
with tangible coefficient such that $m(A)=\alpha(A)$. But $m$ appears in $\beta$ as well by Claim~\ref{clcl1}, so the result follows.
\end{proof}

\begin{claim}\label{clcl2}
There is a polynomial $\rho\in\S[V]$ such that $\alpha=\gamma\oplus\rho$.
\end{claim}

\begin{proof}
By definition of $\gamma$, it is the sum of monomials that appear in $\beta$ with tangible coefficients.
All these monomials appear in $\alpha$ as well by Claim~\ref{clcl1}.
\end{proof}

%
%

\begin{proof}[Proof of Conjecture~\ref{thrConjNiv}]
By Claims~\ref{clcl3} and~\ref{clcl2}, there is $u\in\S$ such that $\alpha(A)=\beta(A)\oplus u$.
Claim~\ref{clcl11} rules out the case when $u$ is tangible and greater than $\beta(A)$.
It remains to note that $\alpha(A)$ and $\beta(A)$ are, respectively, the left-hand and
right-hand sides of the assertion of Conjecture~\ref{thrConjNiv}.
\end{proof}


\begin{thebibliography}{99}
\bibitem{Niv}
A. Niv, On pseudo-inverses of matrices and their characteristic polynomials in supertropical algebra, Linear Algebra Appl. 471 (2015) 264--290.

\bibitem{IRIR2}
Z. Izhakian, L. Rowen, Supertropical algebra, Adv. Math. 225 (2010) 2222--2286.

\bibitem{IRIR}
Z. Izhakian, L. Rowen, A guide to supertropical algebra, Trends in Mathematics, Advances in Ring Theory (2010) 283--302.
\end{thebibliography}
\end{document}